\documentclass{amsart}
\usepackage{amssymb, enumerate, xspace, graphicx}
\usepackage[all]{xy}
\usepackage{pgf}
\usepackage{scrextend}
\SelectTips{cm}{}
\usepackage{tikz}
\usepackage[T1]{fontenc}
\usepackage[final]{microtype}
\usetikzlibrary{arrows,matrix,positioning}
\usepackage{tikz-cd}

\usepackage{filecontents}

\usepackage{mathtools}

\numberwithin{equation}{section}

\numberwithin{subsection}{section}

\allowdisplaybreaks[1]

\newlength{\myarrowsize} 
    \newlength{\myoldlinewidth}
\pgfarrowsdeclare{myto}{myto}{
        \pgfsetdash{}{0pt} 
        \pgfsetbeveljoin 
        \pgfsetroundcap 
        \setlength{\myarrowsize}{0.6pt}
        \addtolength{\myarrowsize}{.5\pgflinewidth}
        \pgfarrowsleftextend{-4\myarrowsize-.5\pgflinewidth} 
        \pgfarrowsrightextend{.7\pgflinewidth}
    }{
        \setlength{\myarrowsize}{0.6pt} 
        \addtolength{\myarrowsize}{.5\pgflinewidth}  
        \setlength{\myoldlinewidth}{\pgflinewidth}
        \pgfsetroundjoin
        \pgfsetlinewidth{0.0001pt}
        \pgfpathmoveto{\pgfpoint{0.43\myarrowsize}{0}}
        \pgfpatharc{0}{70}{0.14\myarrowsize}
        \pgfpatharc{-110}{-169.5}{4\myarrowsize}
        \pgfpatharc{10.5}{189}{0.25\myarrowsize and 0.12\myarrowsize}
        \pgfpatharc{-170}{-119.5}{4.48\myarrowsize}
        \pgfpathmoveto{\pgfpoint{0.43\myarrowsize}{0}}
        \pgfpatharc{0}{-70}{0.14\myarrowsize}
        \pgfpatharc{110}{169.5}{4\myarrowsize}
        \pgfpatharc{-10.5}{-189}{0.25\myarrowsize and 0.12\myarrowsize}
        \pgfpatharc{170}{119.5}{4.48\myarrowsize}
        \pgfpathclose
        \pgfsetstrokeopacity{0.25}
        \pgfusepathqfillstroke
    }

\pgfarrowsdeclare{mydashto}{mydashto}{
        \pgfsetdash{{3pt}{3pt}}{0pt} 
        \pgfsetbeveljoin 
        \pgfsetroundcap 
        \setlength{\myarrowsize}{0.6pt}
        \addtolength{\myarrowsize}{.5\pgflinewidth}
        \pgfarrowsleftextend{-4\myarrowsize-.5\pgflinewidth} 
        \pgfarrowsrightextend{.7\pgflinewidth}
    }{
        \setlength{\myarrowsize}{0.6pt} 
        \addtolength{\myarrowsize}{.5\pgflinewidth}  
        \setlength{\myoldlinewidth}{\pgflinewidth}
        \pgfsetroundjoin
        \pgfsetlinewidth{0.0001pt}
        \pgfpathmoveto{\pgfpoint{0.43\myarrowsize}{0}}
        \pgfpatharc{0}{70}{0.14\myarrowsize}
        \pgfpatharc{-110}{-169.5}{4\myarrowsize}
        \pgfpatharc{10.5}{189}{0.25\myarrowsize and 0.12\myarrowsize}
        \pgfpatharc{-170}{-119.5}{4.48\myarrowsize}
        \pgfpathmoveto{\pgfpoint{0.43\myarrowsize}{0}}
        \pgfpatharc{0}{-70}{0.14\myarrowsize}
        \pgfpatharc{110}{169.5}{4\myarrowsize}
        \pgfpatharc{-10.5}{-189}{0.25\myarrowsize and 0.12\myarrowsize}
        \pgfpatharc{170}{119.5}{4.48\myarrowsize}
        \pgfpathclose
        \pgfsetstrokeopacity{0.25}
        \pgfusepathqfillstroke
    }

\newtheorem*{namedtheorem}{\theoremname}
\newcommand{\theoremname}{testing}

\newtheorem{theo}{Theorem}[section]
\newtheorem{prop}[theo]{Proposition}
\newtheorem{prop-def}[theo]
{Proposition-Definition}

\newtheorem{lemm}[theo]{Lemma}
\newcommand \fr{\operatorname{F}}
\theoremstyle{definition}
\newtheorem{defi}[theo]{Definition}

\newtheorem{rema}[theo]{Remark}

\newtheorem{conj}[theo]{Conjecture}

\theoremstyle{remark}

\newcommand\GG{\mathbb{G}}

 \newcommand\ZZ{\mathbb{Z}}

\newcommand\mto{\ifinner\mapsto\else\longmapsto\fi}
\newcommand\too{\longrightarrow}

\newcommand \In{\subseteq}

\renewcommand\H{\operatorname{H}}

\def\displaytimes_#1{\mathrel{\mathop{\times}\limits_{#1}}}

\def\displayotimes_#1{\mathrel{\mathop{\bigotimes}\limits_{#1}}}

\renewcommand\hom{\operatorname{Hom}}

\newcommand\spec{\operatorname{Spec}}

\newcommand{\ed}{\operatorname{ed}}

\newcommand{\trdeg}{\operatorname{tr\,deg}}

\begin{document}

\title{Essential dimension of inifinitesimal unipotent  group schemes}

\author{Dajano Tossici}

\address{Institut de Math\'ematiques de Bordeaux, 351 Cours de la Liberation\\
33 405 Talence \\ France}
\email[Tossici]{Dajano.Tossici@math.u-bordeaux1.fr}

%\date{\arrday}

%\tableofcontents

% \begin{resume} résumé en français \end{resume}
%    \english
% \begin{abstract} the same in english \end{abstract}
%    \french
% % 
 \begin{abstract}
 We propose a generalization of Ledet conjecture, which predicts the essential dimension of cyclic $p$-groups in characteristic $p$, for finite commutative unipotent group schemes.  And we show some evidence and some consequences of this new conjecture. 
 \end{abstract}

%  \renewcommand{\abstractname}{R\'esum\'e}
%    \begin{abstract}
%      Nous proposons une g\'en\'eralisation de la Conjecture de Ledet, qui pr\'edit la dimension essentielle de $p$-groupes cycliques en caract\'eristique $p$, pour sch\'emas en groupes finis commutatifs unipotents.
% Et nous montrons quelques \'evidences et  consequences de cette nouvelle conjecture.
%    \end{abstract}

\maketitle

\section{Introduction}

The notion of essential dimension of a finite group over a field $k$ was introduced by Buhler and Reichstein (\cite{BR}). It was
later extended to various contexts. First Reichstein
generalized it to linear algebraic groups (\cite{Re}) in
characteristic zero; afterwards Merkurjev gave a general definition for functors from the
category of extension fields of the base field $k$ to the category of sets (\cite{BF}).
In particular one can consider the essential dimension of group schemes over a field (see Definition \ref{def: ed usuale}).
% Brosnan, Reichstein and Vistoli (\cite{BRV})
%studied the essential dimension of  algebraic stacks, a general class which
%includes almost all the examples of interest.

%Important results  on the essential dimension of finite groups in
%characteristic $0$ have been proved by Florence (\cite{Flo}) and
%Karpenko and Merkurjev (\cite{KM}). The essential dimension of
%finite groups in positive characteristic has been studied by Ledet
%\cite{Le}. As to higher dimensional groups, there are works of
%Reichstein and Youssin (\cite{Rei-You}), Chernousov and Serre
%(\cite{Che-Ser}), Gille and Reichstein \cite{Gil-Rei}, Brosnan,
%Reichstein and Vistoli \cite{BRV2},   on algebraic groups, Brosnan
%\cite{Bro} on abelian varieties over $\mathbb{C}$, and Brosnan and
%Shreekantan \cite{brosnan-shreekantan08} on abelian varieties over
%number fields.
%In \cite{TV} it has been investigated the case of non-smooth group schemes (necessarily in positive
%characteristic). 

If $G$ is a flat group scheme  of locally finite presentation  over a scheme $S$, a
 $G$-torsor over $X$ is an $S$-scheme $Y$ with a left $G$-action by $X$-automorphisms  and a  faithfully flat and locally of finite presentation  morphism $Y\to X$ over $S$  such that the map $G\times_S Y\to Y\times_X Y$ given by $(g,y)\mapsto (gy,y)$ is an isomorphism.
 We recall that isomorphism classes of $G$-torsors over $X$ are classified by the pointed set $\H^1(X,G)$ if $G$ is affine (\cite[III, Theorem 4.3]{Mi}) or $G$ an abelian scheme and $X$ is regular \cite[Proposition XIII 2.6]{R}. We will restrict to these two cases.
If $G$ is commutative, then $\H^{1}(X, G)$ is a
group, and coincides with the cohomology group of $G$ in the fppf
topology.

\begin{defi}\label{def: ed usuale}
Let $G$ be a group scheme of finite type over a field $k$. Let
$k\In K$ be an extension field and $[\xi]\in \H^1(\spec(K),G)$ the
class of a $G$-torsor $\xi$. Then the essential dimension of $\xi$
over $k$, which we denote by $\ed_k\xi$, is the smallest
non-negative integer $n$ such that 
\begin{itemize}
\item[(i)]
there exists a subfield $L$ of
$K$ containing $k$, with $\trdeg(L/k) = n$,  
\item[(ii)]
such that $[\xi]$ is
in the image of the morphism 
$$
\H^1(\spec(L),G)\too
\H^1(\spec(K),G).$$
\end{itemize}
The essential dimension of $G$ over $k$, which we denote by
$\ed_k G$, is the supremum  of $\ed_k \xi$, where $K/k$ ranges
through all the extension of $K$, and $\xi$ ranges through all the
$G$-torsors over $\spec(K)$.

\end{defi}
%This definition can be extended easily to define essential dimension of functors from the category of field extensions of $k$ to the category of sets (\cite[Definition 1.2]{BF}). %and more generally for algebraic stacks (\cite[Definition 2.2]{BRV}).
We study essential dimension of finite commutative unipotent group scheme over a field $k$ of positive characteristic $p$.
We recall the following conjecture due to Ledet  (\cite{Le}).
%{
%\renewcommand{\thetheo}{\ref{conj Ledet}}
\begin{conj}\label{conj Ledet}
The essential
dimension of the cyclic group of order $p^n$ over $k$ is $n$.
\end{conj}
%}
Here we propose a generalization of this conjecture. 
For any commutative group scheme $G$ over $k$ one can define a morphism $V:G^{(p)}\to G$, where $G^{(p)}$ is the fiber product of the morphism $G\to \spec(k)$ and the absolute Frobenius $\spec(k)\to \spec(k)$. This morphism is called Verschiebung. See \cite[IV,\S 3, $n^o$ 4]{DG} for the definition. 
We remark that it can be defined also as the dual of the relative Frobenius $\fr:G^{\vee}\to {G^{\vee}}^{(p)}$ where $G^{\vee}$ is the Cartier dual of $G$.

%\begin{defi}
%Let $G$ be a commutative unipotent group scheme over $k$. We call $V$-\textit{order} for $G$ the minimal  integer $n\ge 1$ such that $V^n= 0$. We note it by $n_V(G)$.
%\end{defi}
%
%%Let $G$ be a commutative group scheme over a field of caracteristic positive. 
%This number exists since  $G$ is unipotent.

\begin{defi}
Let $G$ be a commutative unipotent group scheme over $k$. We call $V$-\textit{order} for $G$ the minimal  integer $n\ge 1$ such that $V^n= 0$. We note it by $n_V(G)$.
\end{defi}

%Let $G$ be a commutative group scheme over a field of caracteristic positive. 
This number exists since  $G$ is unipotent.

\begin{conj}\label{conj1}
Let $k$ be a field of positive characteristic and let $G$ be a finite unipotent commutative group
  scheme.  Then $\ed_k G\ge n_V(G)$.
  %If $V_G^{n-1}$ is not trivial then $\ed_k G\ge n$, or equivalently if $\ed_k G=n$ then $V^n_G=0$.
\end{conj}

%{
%\renewcommand{\thetheo}{\ref{conj1}}
%\begin{conj}\label{conj1}Let $k$ be a field of positive characteristic and let $G$ be a finite unipotent commutative group
%  scheme.  Then $\ed_k G\ge n_V(G)$.
%  where $n_V(G)$ is the minimal integer $n\ge 1$ such that $V^n= 0$
%  \end{conj}
%}

In fact it is easy to see that Conjecture \ref{conj Ledet} is equivalent to Conjecture \ref{conj1} in the case of finite commutative \'etale group schemes (Lemma \ref{lemm:Ledet equivalent}).

We  give some evidences and consequences of the conjecture. For instance we prove that the above conjecture is true if $G$ is annihilated by the relative Frobenius (see Proposition \ref{prop:conj height 1}).
Finally we prove in Proposition \ref{prop:ed ab var}, under the assumption that the conjecture is true, that the essential dimension of a nontrivial abelian variety over a field of positive characteristic is $+\infty$. We recall that the same statement is true over number fields \cite[Theorem 2]{BS}. While over an algebraically closed field of characteristic zero it is two times the dimension of the abelian variety \cite[Theorem 1.2]{Bro}.

\subsection*{Acknowledgements} I would like to thank  A.~Vistoli   for useful comments and conversations. 
I have been partially supported by the project ANR-10-JCJC 0107 from the
Agence Nationale de la Recherche.
This work was partly elaborated   during a stay  at Max Planck Institute of Bonn.

\section{Essential dimension of finite unipotent commutative group schemes and abelian varieties in positive characteristic}
In the following $k$ is a field of positive characteristic $p$.
%We recall the following conjecture due to Ledet  (\cite{Le}).
%\begin{conj}\label{conj Ledet}
%The essential
%dimension of the cyclic group of order $p^n$ over $k$ is $n$.
%\end{conj}
%
%Here we propose a generalization of this conjecture which maybe makes it more clear.
%We recall that for any commutative group scheme $G$ over $k$ one can define a morphism $V:G^{(p)}\to G$, where $G^{(p)}$ is the fiber product of the morphism $G\to \spec(k)$ and the absolute Frobenius $\spec(k)\to \spec(k)$. This morphism is called Verschiebung. See \cite[IV,\S 3, $n^o$ 4]{DG} for the definition. We remark that it can be defined also as the dual of the relative Frobenius $\fr:G^{\vee}\to {G^{\vee}}^{(p)}$ where $ G^{\vee}$ is the Cartier dual of $G$.
%
%\begin{defi}
%Let $G$ be a commutative unipotent group scheme over $k$. We call $V$-\textit{order} for $G$ the minimal  integer $n\ge 1$ such that $V^n= 0$. We note it by $n_V(G)$.
%\end{defi}

%Let $G$ be a commutative group scheme over a field of caracteristic positive. 
%This number exists since  $G$ is unipotent.
%
%\begin{conj}\label{conj1}
%Let $k$ be a field of positive characteristic and let $G$ be a finite unipotent commutative group
%  scheme.  Then $\ed_k G\ge n_V(G)$.
%  %If $V_G^{n-1}$ is not trivial then $\ed_k G\ge n$, or equivalently if $\ed_k G=n$ then $V^n_G=0$.
%\end{conj}
The following Lemma proves that in fact Ledet Conjecture is just a particular case of this conjecture.

\begin{lemm} \label{lemm:Ledet equivalent}
Ledet conjecture is equivalent to Conjecture \ref{conj1} restricted to finite unipotent commutative \'etale group schemes.
\end{lemm}
\begin{proof}

In fact the Verschiebung of $\ZZ/p^n\ZZ$ is just multiplication by $p^n$. So, since 
 $p^{n-1}$ is not trivial over $\ZZ/p^n\ZZ$, then by
the above conjecture we have that $\ed_k\ZZ/p^n\ZZ\ge n$.
On the other hand since $\ZZ/p^n\ZZ$ is contained in the special group of Witt vectors
of length $n$, which has dimension $n$, then $\ed_k\ZZ/p^n\ZZ\le n$.

Conversely, let us suppose Ledet conjecture is true. Let $G$  be a  finite commutative unipotent \textit{\'etale} group scheme over a field $k$ of positive characteristic. Since essential dimension does not increase by field extension (\cite[Prop 1.5]{BF}) we can assume $k$ algebraically closed. So we have that $G$ is the product of cyclic $p$-groups. Let us take a direct summand of order maximum $p^n$. Then $n$ is such that $V^n=0$ but $V^{n-1}\neq 0$. We have that $\ed_kG\ge \ed_k\ZZ/p^n\ZZ=n$, where the last equality follows from Ledet Conjecture. So we are done. 
\end{proof}
\begin{rema}\label{rema:Ledet equivalent level n}
 In fact we have proved slightly more: Ledet conjecture for a fixed $n$ is equivalent to  Conjecture \ref{conj1} restricted to finite unipotent commutative \'etale group schemes with $V$-order equal to $n$.
\end{rema}

%\begin{lemm}
%If the Conjecture \ref{conj1} is true then the Ledet Conjecture is true.
%\end{lemm}
%\begin{proof}
%
%\end{proof}

Some cases of Conjecture \ref{conj1} can be proved using results of \cite{TV}. 
For instance we have the following proposition which treats the case orthogonal to the \'etale case. 

\begin{prop}\label{prop:conj height 1}
The conjecture \ref{conj1} is true for finite unipotent commutative group schemes of height $1$ (i.e. annihilated by Frobenius).
\end{prop}
\begin{proof}
In \cite[Theorem 1.2]{TV} it has been proved that for a finite group scheme the essential dimension is greater than or equal the dimension of its Lie Algebra.
In the case of the proposition the order of the group scheme is $p^n$, where $n$ is the dimension of the Lie Algebra. 

%Moreover, supposing the base field algebraically closed,
%$G$ is an extension of $n$-copies of $\alpha_p$. 
Since, by the lemma below, we have that $V^n=0$ the conjecture is proven in this case. 
\end{proof}
\begin{lemm} Let $k$ be a field of characteristic $p$.
The operator $V^n$ is trivial over any unipotent commutative group scheme $G$ of order $p^n$.  
\end{lemm}
\begin{proof}
We consider $V^n$ as morphism $G^{(p^n)}\to G$.
%Since 
%We can suppose that the field is algebraically closed. 
%Then $G$ is successive extension of $\alpha_p$ and $\ZZ/p\ZZ$. 
%We prove the result by induction. For $n=1$ it is clear since $G$ is $\alpha_p$ or $\ZZ/p\ZZ$. Now suppose that the Lemma is true for group schemes of order smaller or equal to $n$. But  
Since $G$ is unipotent the kernel of $V$ is not trivial, so in particular the image of $V$ has order strictly less than $n$. Iterating the argument and applying it to the subgroup image, we have that the image of $V^{i+1}$ is strictly contained in the image of $V^i$, for any $i\ge 0$. So after $n$-iteration the image is trivial.
 \end{proof}

%\begin{prop} The following statement are equivalent.
%\begin{itemize}
%\item[(i)] The conjecture \eqref{conj1} is true.
%\item[(ii)] The essential dimension of a finite commutative unipotent group scheme is equal to $\min\{\sum_i^rn_i: G\In \Prod W_{n_1,k}\}$.
%\end{itemize}
%\end{prop}
%\begin{proof}
%$(ii)$ implies $(i)$ is ok. The other direction ?
%\end{proof}

Here some easy considerations about the conjecture.

\begin{lemm}\label{lemm:epi mono conj} Let $k$ be a field of positive characteristic and let $G_1$ and $G_2$ be two finite commutative unipotent group schemes over $k$.
\begin{itemize}
\item[(i)] Let $f:G_1\to G_2$ be an epimorphism  (resp. monomorphism) of group schemes with $n_V(G_1)= n_V(G_2)$. If the Conjecture \ref{conj1} is true for $G_2$ (resp. $G_1$) it is true for $G_1$ (resp. $G_2$).
%\item[(ii)] If $f:G_1\to G_2$ is  monomorphism of group schemes and $n_V(G_1)= n_V(G_2)$ then if the Conjecture \ref{conj1} is true for $G_1$ it is true for $G_2$.
\item[(ii)] If the Conjecture \ref{conj1} is true for $G_1$ and $G_2$ then it is true for $G_1\times G_2$. 

\item[(iii)] It is sufficient to prove the Conjecture \ref{conj1} under the
following assumptions 
\begin{itemize}
\item[(1)] $k$ is algebraically closed;
\item[(2)] $G$ is contained in $W_{n,k}$, Witt vectors (of length $n$) group scheme, where $n=n_V(G)$.
\end{itemize}

\item[(iv)] It is sufficient to prove the conjecture for \'etale finite group schemes and infinitesimal group schemes.
\end{itemize}
\end{lemm}

\begin{rema} It is easy to prove that if $f$ is an epimorphism (resp. monomorphism) then one always has $n_V(G_1)\ge n_V(G_2)$ (resp. $n_V(G_1)\le n_V(G_2)$).
\end{rema}
\begin{proof}
\begin{itemize}
\item[(i)] Let us suppose $f$ is an epimorphism. Then
$$
0\too \ker f \too G_1\too G_2\too 0 
$$
is exact. Then, since $\ker f$ is unipotent commutative, for any extension $K$ of $k$  we have that $H^2(\spec(K),\ker f)=0$
 (\cite[Lemma 3.3]{TV}).  
 So $H^1(\spec(K),G_1)\to H^1(\spec(K),G_2)$ is surjective. This implies, by  \cite[Lemma 1.9]{BF}, that $\ed_k G_1\ge \ed_k G_2$. 
So we have
$$
\ed_k G_1\ge \ed_k G_2\ge n_V(G_2)=n_V(G_1)
$$
and we are done.

If $f$ is a monomorphism it is even easier.  In fact we have $\ed_K G_1\le \ed_K G_2$ by \cite[Theorem 6.19]{BF} and so we can conclude as above, switching $G_1$ with $G_2$.
\item[(ii)] It is sufficient to remark that the $V$-order  of $G_1\times G_2$ is the maximum between the $V$-order of $G_1$ and that one of $G_2$. So the result comes from (i) using as morphism the projection over the group scheme with greater $V$-order.
\item[(iii)]
%As remarked above it
%follows from \cite[Prop 1.5]{BF} 
The essential dimension does not
increase after a  base change by \cite[Proposition 1.5]{BF}.
So we assume $k$ algebraically closed.
Let $n=n_V(G)$. 
Since $V_G^n=0$,   by \cite[V,\S1, Proposition 2.5]{DG}, there exists $r$ such that $G$ is contained in
$W_{n,k}^r$. For any $i=1,\dots, r$ let $p_i$ the projection of
$W_{n,k}^r$ on the $i^{th}$ components and let $G_i=p_i(G)$. Since $V^{n-1}$ is different from zero  and the Verschiebung is compatible with morphisms  there exists
an $i_0$ such that $V^{n-1}$ is different from $0$ over $G_{i_0}$.
%, otherwise any
%$G_i$ would be contained in $W_{n-1,k}$ and so
%$G$ would be contained in $\GG_a^r$ This would imply that Verschiebung
%would be trivial, against assumptions. \footnote{spiegare meglio 2}.
Moreover we have an epimorphism $G\to G_{i_0}$. 
So it follows from part (i) % Lemma \ref{lemm:epi mono conj}(i) 
that if the conjecture is true for $G_{i_0}$ it is true for $G$.
\item[(iv)] Over an algebraically closed field a finite group scheme is the direct product of his \'etale part and his connected (hence infinitesimal) part.
So the statement follows from $(ii)$. 
\end{itemize}

\end{proof}
%\begin{prop}
%It is sufficient to prove the Conjecture \ref{conj1} under the
%following assumptions (in addition to that one of the statement of the
%conjecture) 
%\begin{itemize}
%\item[(i)] $k$ is algebraically closed;
%\item[(ii)] $G$ is contained in $W_{n,k}$, where $n=n_V(G)$.
%\end{itemize}
%\end{prop}
%\begin{proof}
%\begin{itemize}
%\item[(i)] As remarked above it
%follows from \cite[Prop 1.5]{BF} that essential dimension does not
%increase after a  base change \footnote{anche citazione locale}.
%
%% \underline{\textsl{We can suppose $G$ unipotent.}} If $k$ is
%% algebraically closed then $G$ is a direct product of a unipotent group
%% scheme and a diagonalizable group scheme. Since $V$ is 
%
%\item[(ii)] 
%%We first prove that we can assume that 
%%Since for any subgroup $H$ of $G$ we have, by
%%\cite[Cor. 4.3]{Mer},  $\ed_k H\le \ed_k G$, it is sufficient to prove
%%that $G$ has a subgroup $H$ such that $V^n_H=0$.
%%
%%Since  $G$ is
%%unipotent then the Verschiebung is nilpotent. Let $m$ be the integer such that
%%$V_G^{m}\neq 0$ and $V_G^{m+1}\neq 0$. Let us define
%%$H:=V_{G}^{m-n}(G^{(p^{m-n})})$. It is a subgroup scheme of $G$. Since
%%\footnote{spiegare meglio}
%%$$
%%(V_G^{m-n}(G^{(p^{m-n})}))^{(p^n)}\simeq V_{G^{(p^n)}}^{m-n}(G^{(p^m)})
%%$$
%%and the Verschiebung is functorial
%%then
%%$
%%V^n_H
%%$ is zero.
%
%
%

%\end{itemize}
%
%\end{proof}

We now treat the case of finite unipotent group schemes with $n_V(G)=2$. The case with trivial Verschiebung is immediate since finite group schemes have positive essential dimension.

\begin{prop}
The conjecture is true for group schemes $G$ with $n_V(G)=2$ if and only if it is true for the group scheme $\ker(V-\fr^m): W_{2,k}\to W_{2,k}$, for $m\ge 1$. 
\end{prop}
\begin{proof}
Clearly we have only to prove the \textit{if part}.
Using Lemma \ref{lemm:epi mono conj}(iv), we have to prove the conjecture only for infinitesimal group schemes since the \'etale case, which we can reduce to $G=\ZZ/p^2\ZZ$ by Remark \ref{rema:Ledet equivalent level n},  is known (\cite[Proposition 7.10]{BF}). In particular the Frobenius is not injective.
By the above Lemma we can suppose that $k$ is algebraically closed and we can suppose that $G$ is contained in $W_{2,k}$.
We have to prove that the essential dimension is at least $2$.
Let us consider the exact sequence
$$
0\too \ker \fr \too G\too \fr(G) \too 0.
$$
%If the Lie Algebra of $\ker \fr$ has dimension strictly bigger than one (and so necessarily $2$ since $\ker \fr \In W_{2,k}$) then by \cite[Theorem 1.2]{TV} we have that $\ed_k \ker \fr \ge 2$ and so we are done since $\ed_k G\ge \ed_k \fr$.
If $V(\ker \fr)\neq 0$ then we are done by Lemma \ref{lemm:epi mono conj}(i) and Proposition \ref{prop:conj height 1}.
If $V(\fr (G))\neq 0$ we are reduced to prove the conjecture for $\fr(G)$, again by Lemma \ref{lemm:epi mono conj}(i).
If $m$ is the smallest integer such that $\fr^m(G)=0$ then 
iterating the argument (at most $m-1$ times) we finally have two possibilities: 
we have to prove the conjecture for a group scheme with $V(\ker \fr)=V(\fr (G))=0$ or for a group annihilated by $\fr$ (this happens if we have to iterate exactly $m-1$ times). The second case is already known by Proposition \ref{prop:conj height 1}. Therefore we can suppose we are in the first case. Therefore $\ker \fr$ and $\fr(G)$ are contained in $\GG_{a,k}$. This means that $\ker \fr =\alpha_{p,k}$ and $\fr(G)=\alpha_{p^m,k}$ for some $m$.
So we can suppose that we have an exact sequence
$$
0\too \alpha_{p,k} \too G\too \alpha_{p^m,k} \too 0.
$$

%We also remark that we can suppose that $\dim_k \Lie G=1$ otherwise the essential dimension is two by \cite[Theorem 1.2]{TV}.

Now we have that $V$ and $\fr^{m}$ induce a morphism from $G$ to $\alpha_{p,k}$. Moreover
$$
\hom_k(G,\alpha_{p,k})=\hom_k(G/F(G^{(1/p)}),\alpha_{p,k}).
$$
But $G$ has order $p^{m+1}$ and $F(G^{(1/p)})$ has order $p^m$, so
$G/F(G^{(1/p)})\simeq \alpha_{p,k}$.
Then $$\hom_k(G,\alpha_{p,k})=\hom_k(\alpha_{p,k},\alpha_{p,k})=k.$$ This implies that
$$
V=a\fr^{m}
$$
for some $a\in k$. Therefore $G$ is contained in $\ker(V-a\fr^m):W_{2,k}\to W_{2,k}$. Since these two group schemes have the same order, as it is easy to check, they are equal.

Finally we remark that since $k$ is algebraically closed it is straightforward to prove that, if we fix $m$, all these groups are isomorphic between them.

\end{proof}
If $m=1$ and $k$ is algebraically  closed the group scheme in the Proposition is nothing else that the $p$-torsion group scheme of a supersingular elliptic curve.

Finally we give a consequence of the conjecture for group scheme of essential dimension one.
\begin{prop}
Let us suppose that conjecture \ref{conj1} is true for group schemes with $V$-order $2$. 
\begin{itemize}
\item[(i)] If $k$ is algebraically closed, a finite   commutative unipotent group scheme has essential dimension $1$ if and only if it is isomorphic to $\alpha_{p^m,k}\times (\ZZ/p\ZZ)^r$ for some $m,r>0$.
\item[(ii)] If a commutative unipotent group scheme has essential dimension $1$ then it is a twisted form of $\alpha_{p^m,k}\times (\ZZ/p\ZZ)^r$ for some $m,r>0$.
\item[(iii)] An infinitesimal commutative unipotent group scheme over a perfect field has essential dimension $1$ if and only if it is isomorphic to $\alpha_{p^m,k}$ for some $m>0$.
\end{itemize}
\end{prop}

\begin{proof}
Clearly (i) implies (ii).
We now prove (i). Let $k$ be algebraically closed. One has just to prove the \textit{only if} part since  $\alpha_{p^m,k}\times (\ZZ/p\ZZ)^r$ is contained in $\GG_{a,k}$ if $k$ is algebraically closed . 
First of all we remark that if the essential dimension is $1$ then the Verschiebung is trivial. 
In fact if $n_G(V)>1$ then $V^{n_G(V)-2}(G^{(p^{n_G(V)-2})})$ is a subgroup scheme of $G$ with $V$-order $2$. 
Therefore if  the conjecture is true for group schemes with $V$-order $2$ then $G$ would have essential dimension strictly greater than $1$.
So  $V=0$. 
Then by  \cite[IV \S3, Corollaire 6.9]{DG} we have that $G$ is isomorphic to $\prod_{i=1}^l\alpha_{p^{n_i},k}\times  (\ZZ/p\ZZ)^r$ for some $l,n_i,r>0$.
Since the essential dimension of $G$ is one then, by \cite[Theorem 1.2]{TV}, we have $\dim Lie(G)\le 1$.
So $G$ is isomorphic to $\alpha_{p^m,k}\times (\ZZ/p\ZZ)^r$ for some $m,r>0$

Using (i), to prove (iii) we have just to 
% 
% Then we have that  there is an inclusion $j:G\to \GG_{a,k}^m$ for some $m>0$. Let us take $m$ minimal. Let us suppose that $m \ge 2$ and let $\pi_i:G\to \GG_{a,k}^{m-1}$ be the composition between $j$ and the map which forgets the $i$-th components. 
% By the minimality of $m$ we have that all $\ker(\pi_i)$ are nonzero.
% Moreover it is clear that 
% $$
% \ker(\pi_i)\cap \ker(\pi_j)=0.
% $$
% if $i\neq j$.
% % Let $m_0$ be the maximum integer such that
% % $$
% % \ker(\pi_{m_0})\neq 0.
% % $$
% %We remark that $1\le m_0 \le m-1$. Let us suppose that 
% %$\ker(\pi_{m_0})\neq 0$. 
% 
% Then we have that $G$ contains $\ker(\pi_i)\cap \ker(\pi_j)$  if $i\neq j$.
% %$\cap_{i=1}^{m=0} \ker(\pi_i)\times \ker(\pi_{m_0})$. 
% %So $G$ is isomorphic to a subgroup of $\GG_{a,k}^{m-1}$. But this contradicts the minimality of $m$. 
%So $m=1$ and then $G$ is isomorphic to $\alpha_{p^n,k}$ for some $n$ (\cite{h}).
show that $\alpha_{p^m,k}$ has no nontrivial twisted forms. 
This follows from the fact that $Aut_{\bar{k}}(\alpha_{p^n,\bar{k}})$ is isomorphic to $\GG_{m,\bar{k}}\times \GG_{a,\bar{k}}^{m-1}$. 
Since twisted forms are classified by the first cohomology of this group then there is only one twisted form.

\end{proof}
\begin{rema}
 For \'etale group schemes the result is unconditional since the conjecture is true.
 This particular case can be deduced by the case of constant group schemes (\cite[Proposition 5 and 7]{Le1}).
\end{rema}

Finally we give a consequence of the conjecture for the essential dimension of abelian varietis in positive characteristic.
\begin{prop}\label{prop:ed ab var}
Let $A$ be a nontrivial abelian variety  over a field  $k$ of characteristic $p>0$. If the Conjecture \ref{conj1} is true then
$$
\ed_k A=+\infty.
$$
\end{prop}
\begin{proof}
As usual we can  suppose that $k$ is algebraically closed.
For any positive integer $n$ we call $A[p^n]$ the group scheme of $p^n$-torsion of $A$.
Since $A[p^n]\In A$, by \cite[Principle 2.9]{Bro}\footnote{We remark that the proof of \cite[Theorem 6.19]{BF}, which we used before and which is the standard reference for this result, works only for affine group schemes since it uses versal torsors. We remark that in \cite{BF} an \textit{algebraic group} is intended to be affine.}, we have that
\begin{equation}\label{eq:ed Apn}
\ed_k A[p^n]\le \ed_k A + \dim A.
\end{equation} 
Now by \cite[pag. 147]{M} we have that there exists an integer $r\ge 0$ such that
$$
A[p^n]\simeq (\ZZ/p^n\ZZ \times \mu_{p^n})^r\times G^0_n
$$ 
where $G^0_n$ is unipotent and infinitesimal.
If $r=0$ then we remark that $V^{n-1}$ is not trivial over 
$G^0_n=A[p^n]$ otherwise $A[p^n]=A[p^{n-1}]$.
If $r>0$ then $V^{n-1}((\ZZ/p^n\ZZ)^r)\neq 0$.
So, if we call $U_n$ the unipotent part of $A[p^n]$
%And we remark that the order of $U_n$ is larger than or equal to $p^{n\dim A}$. Indeed we know that  the order of $A^\vee[p^n](k)$, where $A^\vee$ is the dual variety of $A$, is less than or equal to $p^{n\dim A^\vee}=p^{n\dim A}$. But the order of  $A^\vee[p^n](k)$ is exactly the order of $D_n$. Now we remark that
%$
%\ker {V^{n-1}}\In A[p^{n-1}]
%$
%so its order is less than or equal $p^{2(n-1)\dim A}$. In 
we have that $V^{n-1}$ is not trivial over $U_n$. So by  conjecture \ref{conj1}, and \cite[Theorem 6.19]{BF}, we have that 
$$
\ed_k A[p^n]\ge \ed U_n \ge n.
$$
which, together with \eqref{eq:ed Apn}, gives $\ed_k A=+\infty$.

\end{proof}

\bibliographystyle{amsalpha}
\bibliography{essdimdvr_weak4}

\end{document}